\documentclass[10pt]{amsart}

\usepackage{amsfonts,amsthm,amsmath}

            \def\es{\emptyset}

\def\ZZ{\mathbb{Z}}
\def\CC{\mathbb{C}}
\def\RR{\mathbb{R}}

\newcommand{\diam}{{\rm diam}\,}
\newcommand{\Crit}{{\rm Crit}\:}

\newtheorem{theo}{Theorem}
\newtheorem{lem}{Lemma}
\newtheorem{fact}{Fact}
\newtheorem{prop}{Proposition}

            \def\es{\emptyset}

\title{No Finite Invariant Density  for Misiurewicz Exponential Maps}

\author{Janina Kotus}
\address{Faculty of
 Mathematics and Information Science,
Warsaw University of Technology, 00-661 Warsaw, Poland}
\email{J.Kotus@impan.gov.pl}

\author{ Grzegorz \'{S}wia\c\negthinspace tek } \address{Department  of
Mathematics, Penn State University, University Park, PA 16802, USA}
\email{swiatek@math.psu.edu}

\thanks{MSC 37F10\\
The first author is partially supported by a grant {\it Chaos, fraktale i dynamika
konforemna} - N N201 0222 33}

\begin{document}

\begin{abstract}
For exponential mappings such that the orbit of the only singular
value $0$ is bounded, it is shown that no integrable density invariant
under the dynamics exists on $\CC$.  
\end{abstract}

\maketitle

\section{Introduction}
We consider one parameter  family of exponential functions
$f_\Lambda(z)=\Lambda e^z$, $z\in \mathbb{\CC}$, $ \Lambda  \in
\CC^*$. These maps have only one finite  singular value $0$ whose
forward trajectory determines the dynamics on $\CC$. From now we
assume that the orbit of the  asymptotic value $0$ is bounded, hence
the Julia set $J(f_\Lambda)=\CC$. Thus $f$ satisfies so called
Misiurewicz condition  i.e. the post-singular set
$P(f):=\overline{\bigcup_{n=0}^\infty f^n_\Lambda(0)}$ is bounded
and $P(f) \cap \Crit(f)=\es$. It follows from {\cite[Th.1]{GKS}}
that $P(f)$ is hyperbolic. The problem of existence of probabilistic
invariant measure absolutely continuous with respect to the Lebesgue
measure (abbr. \emph{pacim}) for transcendental meromorphic
functions satisfying Misiurewicz condition was discussed in
\cite{KS1}. However this  result cannot be applied to entire
functions. It is still an open problem whether the  simplest entire
functions like exponential map $z\to 2\pi i \exp(z) $ have
\emph{pacim}. The main result of this paper is the following
theorem.
\begin{theo}\label{theo:10zp,1}
Let $f(z)=\Lambda\exp(z)$ with $\Lambda\in \CC \setminus \{0\}$ chosen so that
the Julia set is the entire sphere and the orbit of $0$ under $f$ is bounded.
Then $f$ admits no probabilistic invariant measure absolutely
continuous with respect to the Lebesgue measure.
\end{theo}

However  these  maps have  $\sigma$-finite  invariant measure
absolutely continuous with respect to the Lebesgue measure (see
\cite{KU1}).  A result similar to Theorem~\ref{theo:10zp,1} has been
mentioned to us by other authors, \cite{DS}.

\

 The proof will proceed by contradiction, so we suppose
that such a measure exists and call it $\mu$, while reserving
$\lambda$ for the Lebesgue measure of the plane.  It follows from
\cite{EL} that the set of points escaping to $\infty$ has zero
Lebesgue's measure for  every map in our family. It is not difficult
to prove that for these functions the union $P(f)\cup\{\infty\}$  is
not a metric attractor in sense of Milnor with respect to the
measure $\lambda $ on $\CC$. The results of \cite{Bo} implies that
$f_{\Lambda}$ is ergodic with respect  to $\lambda$. Thus
\begin{fact}\label{fa:10zp,1}
The measure $\mu$ is ergodic.
\end{fact}
\section{Proof}
For a positive integer $n$ write $A_n := \{ z :\: |\Lambda| e^n <
|z| \leq |\Lambda| e^{n+1},\, \arg z \neq \arg \Lambda\}$. A {\em
fundamental rectangle} will refer to any set in the form $\{x+2\pi
iy :\: k<x<k+1, l<y<l+1\}$ for integers $k,l$. Thus, any fundamental
rectangle is mapped with bounded distortion and onto some $A_n$.
\begin{lem}\label{lem:10zp,1}
For all $n\in \ZZ_+$, $\inf{\rm ess} \{\frac{d\mu}{d\lambda}(z) :\:
z\in A_n \} > 0$.
\end{lem}
\begin{proof}
By~{\cite[Th.1]{GKS}},  the post-singular set $P(f)$ has area $0$,
so it cannot be the support of $\mu$. Additionally, the image of
every open set covers $A_n$ after finitely many iterations, so it
suffices to have the $\frac{d\mu}{d\lambda}$ essentially bounded
away from $0$ on any open set. Hence, Lemma~\ref{lem:10zp,1} follows
from the following fact.
\end{proof}
\begin{lem}\label{lem:13zp,2}
Suppose that $F$ is a meromorphic function whose Julia set is the
entire sphere, and $\nu$ a probability invariant and ergodic measure
absolutely continuous with respect to $\lambda$ and such that the
$\nu$-measure of the closure of the post-singular set of $F$ is less
than $1$. Then, there is an open set $U$ such that \[ \inf{\rm ess}
\left\{ \frac{d\nu}{d\lambda}(z) :\: z\in U\right\} > 0.\]
\end{lem}
\begin{proof}
Fix $U$ to be a disk in a positive distance from the orbit of $0$
and such that  $\eta := \nu(U)$ is positive. Denote $\rho(z) :=
\frac{d\nu}{d\lambda}$. Pick $\epsilon>0$. In the argument to follow
it is important distinguish between parameters that do or do not
depend on $\epsilon$.\\

\paragraph{\bf A variant of Luzin's Theorem.}
For every $\epsilon>0$, we can find a continuous
function with compact support
$\rho_{\epsilon} :\: \CC \rightarrow [0,+\infty)$ such that
\begin{equation}\label{equ:13zp,4}
\int_{\CC} (\rho_{\epsilon}(w) -\rho(w))_+\, d\lambda(w) < \epsilon
\end{equation}
where the plus subscript denote the positive part,
\begin{equation}\label{equ:13zp,1}
\int_{\CC} \min (\rho_{\epsilon}(z), \rho(z))\, d\lambda(z) \geq 1 - \eta/10 \; .
\end{equation}
This statement follows from introductory measure theory.\\

\paragraph{\bf Proof of Lemma~\ref{lem:10zp,1} continued.}
Now for any $k$ consider the set $\Omega_k$ of connected components
of $F^{-k}(U)$ which intersect the support of $\rho_{\epsilon}$.
 If $V\in \Omega_k$, then $F^k$ maps $V$ onto $U$
univalently and with distortion bounded depending solely on $U$.
Denote $d_k = \sup \{ \diam V :\: V\in \Omega_k \}$. Since the Julia
set is the whole sphere, $\lim_{k\rightarrow \infty} d_k = 0$. Let
$G_k$ denote the set of inverse branches of $F^k$ defined on $U$.
For $z$ in $U$ \[ \rho_{\epsilon,k}(z) = \sum_{g\in { G}_k} \inf \{
\rho_{\epsilon}(w) :\: w=g(z), z\in U\} |g'(z)|^2\;.\] For any $g$,
the ratio of the values of each summand at two points $z_1, z_2$ is
equal to the ratio of $|g'|^2$ at these points, hence bounded above
by some $Q_0\geq 1$ which depends solely on the distortion of
inverse branches and therefore only on $U$. Consequently,
\begin{equation}\label{equ:13zp,3}
 \frac{\rho_{\epsilon,k}(z_1)}{\rho_{\epsilon,k}(z_2)} \leq Q_0
\end{equation}
for every $z_1, z_2\in U$. Consider a similarly constructed
\[\tilde{\rho}_{\epsilon}(z) =  \sum_{g\in { G}_k}
\rho_{\epsilon}(g(z)) |g'(z)|^2\;.\] By the change of variable
formula
\[ \int_U \tilde{\rho}_{\epsilon}(z)\, d\lambda(z) =
\int_{F^{-k}(U)} \rho_{\epsilon}(w)\, d\lambda(w) \geq
\int_{F^{-k}(U)} \min(\rho(w),\rho_{\epsilon}(w))\, d\lambda(w) = \]
\[ =\int_{\CC} \min(\rho(w),\rho_{\epsilon}(w))\, d\lambda(w) -
\int_{F^{-k}(U)^c} \min(\rho(w),\rho_{\epsilon}(w))\, d\lambda(w) \geq \]
\begin{equation}\label{equ:14za,1}
\geq 1 - \eta/10 - \nu(F^{-k}(U)^c) = 1  - \eta/10 - (1 - \eta) =
\frac{9}{10}\eta
\end{equation}
where we have also used condition~(\ref{equ:13zp,1}). Clearly,
$\rho_{\epsilon,k} \leq \tilde{\rho}_{\epsilon}$. Let
$\delta_{\epsilon}$ denote the modulus of continuity of
$\rho_{\epsilon}$. Then \[\int_U \left( \tilde{\rho}_{\epsilon}(z) -
\rho_{\epsilon,k}(z)\right) \, d\lambda(z) \leq
\delta_{\epsilon}(d_k) \int_U \sum_{g\in { G}'_k} |g'(z)|^2 \,
d\lambda(z)\;.\] Here ${ G}'_k$ denoted the set of only those
inverse branches which map onto some $V\in \Omega_k$. By bounded
distortion, if $g$ maps on $V$, then for any $z\in U$, $|g'(z)|^2
\leq Q_0 \frac{\lambda(V)}{\lambda(U)}$. Hence, we can further
estimate \[\int_U \left( \tilde{\rho}_{\epsilon}(z) -
\rho_{\epsilon,k}(z)\right) \, d\lambda(z) \leq
\delta_{\epsilon}(d_k) \lambda(U)^{-1} \sum_{V\in
  \Omega_k}\lambda(V)\;.\]
Since all $V\in \Omega_k$ must touch the compact support of
  $\rho_{\epsilon}$ and their diameters tend uniformly to $0$ with
  $k$, their joint area remains bounded depending solely on
  $U,\epsilon$. Since also $d_k$ tend to $0$ with $k$, for all $k\geq
  k(\epsilon)$,
\[\int_U \left( \tilde{\rho}_{\epsilon}(z) -
\rho_{\epsilon,k}(z)\right) \, d\lambda(z) \leq \frac{2}{5} \eta \;.
\] Taking into account estimate~(\ref{equ:14za,1}), for $k\geq
k(\epsilon)$, $ \int_U \rho_{\epsilon,k}(z)\, d\lambda(z) \geq
\eta/2\;$. Based on estimate~(\ref{equ:13zp,3}), we conclude that
for all $k\geq k(\epsilon)$,
\begin{equation}\label{equ:13zp,5}
\rho_{\epsilon, k}(z) \geq Q_1 > 0
\end{equation}
for all $z\in U$ and $Q_1$ which only depends on $U$. Next, we
estimate \[\int_U (\rho_{\epsilon,k}(z)-\rho(z))_+\, d\lambda(z)
\leq \int_U (\tilde{\rho}_{\epsilon}(z)-\rho(z))_+\, d\lambda(z) =
\int_{\CC} (\rho_{\epsilon}(w)-\rho(w))_+\, d\lambda(w) < \epsilon\]
where we used a change of variables formula and
condition~(\ref{equ:13zp,4}). For every $\epsilon>0$ and $k\geq
k(\epsilon)$, we conclude from this and estimate~(\ref{equ:13zp,5})
that $ \rho(z) < \frac{Q_1}{2}$ on a set $\lambda$-measure less than
$\frac{2\epsilon}{Q_1}$. Since $\epsilon$ can be made arbitrarily
small while $Q_1$ is fixed, then  $\rho(z) \geq \frac{Q_1}{2}$ on a
set of full $\lambda$-measure in $U$.
\end{proof}
\subsection{Return times}
Introduce the following function $g :\: \RR \rightarrow \RR$: $g(x)
= |\Lambda| \sqrt{e^x}$.
\begin{lem}\label{lem:10zp,2}
There exists $N_0$ such that for all $n\geq N_0$, there exist sets
$W_+, W_-\subset A_n$ which consist of fundamental rectangles each
of which is mapped by $f$ onto some $A_m \subset \{ z\in\CC :\: |z|
\geq g(|\Lambda|e^n)\}$ in the case of $W_+$, $A_m\subset \{ z\in\CC
:\: |z| \leq g(-|\Lambda|e^n)\}$ for $W_-$ and such that
\[\lambda(W_{\pm}) > \frac{1}{4} \lambda(A_n) \;.\]
\end{lem}
\begin{proof}
For an annulus centered at $0$ with inner radius $r$, $1/3$ of its
area belongs to the half-plane $\Re z > r/2$ and another $1/3$ to
$\Re z < - r/2$. For $A_n$ with $n$ large enough, almost the entire
area, certainly more than $1/4$ of the are of the whole annulus, of
$A_n \cap \{ z :\: \Re z > |\Lambda| \exp n  \}$ can be filled with
fundamental rectangles. This defines $W_+$. The set $W_-$ is
constructed in the same way.
\end{proof}
The following lemma generalizes Lemma~\ref{lem:10zp,2}.
\begin{lem}\label{lem:10zp,3}
There are constants $N_1$ and $K_0>1$ such that for all $n\geq N_1$
and any integer $p\geq 1$,  there is a set $W_p \subset A_n$ such
that:
\begin{itemize}
\item
$W_p$ is the union of sets each of which is mapped by $f^{p-1}$
univalently onto a fundamental rectangle,
\item
for every $z\in W_p$ and $0<j<p$, $f^j(z) \in A_m$ with $m\geq n$,
while $f^p(z)\in A_m$ with $m\geq g^p(|\Lambda|e^n)$,
\item
$\lambda(W_p) \geq K_0^{-p}$.
\end{itemize}
\end{lem}
\begin{proof}
Choose $N_1$ at least as large as $N_0$ in Lemma~\ref{lem:10zp,2}
and so large that $g(|\Lambda|e^{N_1}) \geq |\Lambda|e^{N_1}$.
Additionally, the orbit of $0$ must fit inside
$D(0,|\Lambda|\exp(N_1-1))\;$. For $p=1$ the claim follows from
Lemma~\ref{lem:10zp,2}. Assuming now the claim for some $p\geq 1$,
we can first split the set $W_p$ into $W_p^m$, $A_m\subset \{
z\in\CC :\: |z| \geq g^p(|\Lambda|e^n)\}$, defined by $W_p^m = \{
z\in W_p :\: f^p(z) \in A_m\}$. The set $W_p^m$ splits into the
union of topological disks each of which is initially mapped by
$f^{p-1}$ univalently onto a fundamental rectangle and then by $f$.
Since by our choice of $N_1$ the post-singular set is far away
surrounded by $A_{N_1-1}$, the first map has distortion bounded
independently of $m,p$ and the distortion of $f$ satisfies the
explicit bound of $e$. Let $Q>1$ bound the ratio of the squares of
the derivatives for any branch of $f^{-p}$ from $A_m$ into $W_p^m$
at any two points of $A_m$. Since $m\geq N_1 \geq N_0$ by our choice
of $N_1$, inside $A_m$ we can find $W$ given by
Lemma~\ref{lem:10zp,2} and then define $W_{p+1}^m := W_p^m \cap
f^{-p}(W)$. As a consequence of the bounded distortion of $f^p$ and
Lemma~\ref{lem:10zp,2},\, $
\frac{\lambda(W_{p+1}^m)}{\lambda(W_p^m)} \geq (4Q)^{-1} \;$.  Now
set $W_{p+1} = \bigcup_{m\geq g^p(n)} W_{p+1}^m$. Then
$\lambda(W_{p+1}) \geq (4Q)^{-1} \lambda(W_p)$ so with $K_0 = 4Q$
the last claim of Lemma~\ref{lem:10zp,2} will persist under
induction. The remaining claims follow immediately from
Lemma~\ref{lem:10zp,2} and the construction of $W_{p+1}$.
\end{proof}
\begin{prop}\label{prop:13za,1}
There exist constants $N_2$ and $K_0, K_1>1$ such that for each
$n\geq N_2$ and $p\geq 1$, $A_n$  contains a subset $V_p$, such that
$V_p$ are pairwise disjoint for different $p$ and for every $z\in
V_p$, $|f^i(z)| \geq |\Lambda| e^n$  for $i=0,\cdots,p$ while
$|f^{p+1}(z)| \leq g(-g^p(|\Lambda|e^n))$. Additionally, for each
$p$, $\lambda(V_p) \geq K_1^{-1} K_0^{-p} \lambda(A_n).$
\end{prop}
\paragraph{\bf Proof of the Proposition.}
We choose $N_2$ at least equal to $N_1$ from Lemma~\ref{lem:10zp,3},
such that $g(|\Lambda| e^n) \geq |\Lambda| e^n$ if $n\geq N_2$ and
so big that the orbit $0$ fits inside $D(0,|\Lambda|e^{N_2-1})$  and
at least $1$. By the last choice, the pairwise disjointness of sets
$V_p$ will follow automatically from the conditions on orbits from
$V_p$. Consider first the set $W_p$ obtained from
Lemma~\ref{lem:10zp,3}. It consists of sets $U_j$ which are
univalent preimages of fundamental rectangles, each of which is
mapped with bounded distortion onto $A_m \subset \{ z\in\CC :\: |z|
\geq g^p(|\Lambda|e^n)\}$. Thus, a portion of $U_j$ of area at least
$K_1^{-1}\lambda(U_j)$ with $K_1$ a constant, is occupied by the
preimage by $f^p$ of the set $W_-$ from Lemma~\ref{lem:10zp,2}. It
is immediate that every $z$ from this preimage satisfies the demands
of Proposition~\ref{prop:13za,1}. $V_p$ is the union of such
preimages for all $U_j$ and hence its measure is bounded below as
claimed in the Proposition.\\
\paragraph{\bf Proof of Theorem~\ref{theo:10zp,1}.}
\begin{lem}\label{lem:13za,1}
For all $x\geq N_3$ for some $N_3$ and every $\gamma>0$,
$\lim_{p\rightarrow\infty} g^p(x)\gamma^{-p} = +\infty \;$.
\end{lem}
\begin{proof}
Evidently, $g(x)/x$ tends to $\infty$, so pick $N_3$ so that for all
$x\geq N_3$, $g(x) \geq 2x$. Then $g^{p}(x) \geq 2^p x$ for all
$p\geq 1$, in particular $g^p(x) - g^{p-1}(x) \rightarrow +\infty$.
But $\frac{g^{p+1}(x)}{g^p(x)} = \exp(g^p(x)-g^{p-1}(x)) \;$.
\end{proof}
Consider a slit annulus $A_n$ for $n$ at least equal to the constant
$N_2$ of Proposition~\ref{prop:13za,1} and $|\Lambda|e^n \geq N_3$
of Lemma~\ref{lem:13za,1}. Let $\tau(z)$ for $z\in A_n$ be the first
return time to $A_n$. Note that $\mu$-almost every point returns
since open sets return and $\mu$ is ergodic. Clearly $\tau$ is
$\mu$-integrable, but then also $\lambda$-integrable in view of
Lemma~\ref{lem:10zp,1}. Similarly, $\lambda$-almost every point
returns. If $z\in D(0,r)$ then it takes at least $k \geq K_2 \log
r^{-1}$ for $f^k(z)$ to get in the distance at least $1$ away from
the orbit of $0$. $K_2$ is  a positive constant which depends on the
maximum modulus of the derivative of $f$ on some compact set. It
follows that on each set $V_p$ from Proposition~\ref{prop:13za,1},
the return time is at least $K_2(\log |\Lambda|+g^p(|\Lambda|e^n))$.
Since the measure of $V_p$ is only exponentially small with $p$, by
Lemma~\ref{lem:13za,1}, the return time is not $\lambda$-integrable
which gives us the final contradiction.

\end{document}